\documentclass[a4paper,11pt,twoside]{amsart}

\usepackage{graphicx}
\usepackage{amsfonts}
\usepackage{amsmath}
\usepackage{amsthm}
\usepackage{amssymb}
\usepackage[all]{xy}
\usepackage{hyperref}
\usepackage{rotating}

\newtheorem{thm}{Theorem}[section]
\newtheorem{prop}[thm]{Proposition}
\newtheorem{lem}[thm]{Lemma}
\newtheorem{cor}[thm]{Corollary}

\numberwithin{equation}{section}

\theoremstyle{definition}

\newtheorem{remark}[thm]{Remark}
\newtheorem{ex}[thm]{Example}

\DeclareMathOperator{\spec}{Spec} 
\newcommand{\iso}{\cong}
\newcommand{\niso}{\ncong}

\newcommand{\farg}{-} 
\newcommand{\id}{\mathrm{id}}
\newcommand{\dual}{^{\vee}} 
\newcommand{\comp}{\circ} 
\newcommand{\mor}[1]{\xrightarrow{#1}}
\newcommand{\isomor}{\mor{\sim}} 
\newcommand{\rest}[1]{|_{#1}} 

\newcommand{\K}{\Bbbk} 
\newcommand{\cat}[1]{{\mathbf{#1}}} 
\newcommand{\opp}{^{\circ}} 
\newcommand{\s}[1]{\mathcal{#1}} 
\newcommand{\so}{\s{O}} 
\newcommand{\sod}{\s{O}_{\Delta}} 
\newcommand{\Hom}{\mathrm{Hom}}
\newcommand{\Ext}{\mathrm{Ext}}

\newcommand{\Pic}{\mathrm{Pic}} 
\newcommand{\cone}[1]{\mathrm{C}(#1)} 
\newcommand{\sh}[2][1]{#2[#1]} 
\newcommand{\FM}[2][]{\Phi^{#1}_{#2}} 
\newcommand{\D}[1][]{\mathrm{D}^{#1}} 
\newcommand{\Db}{\D[b]} 
\newcommand{\Dp}[1][]{\cat{Perf}_{#1}} 
\newcommand{\rd}{\mathbf{R}} 
\newcommand{\ld}{\mathbf{L}} 
\newcommand{\lotimes}{\overset{\ld}{\otimes}} 
\newcommand{\ds}{\omega} 
\newcommand{\fun}[1]{\mathsf{#1}} 
\newcommand{\Coh}{\cat{Coh}} 
\newcommand{\Qcoh}{\cat{Qcoh}} 
\newcommand{\p}{\mathrm{p}} 
\newcommand{\Ker}{\s{E}} 
\newcommand{\ExFun}{\cat{ExFun}} 

\newcommand{\cal}{\mathcal}
\newcommand{\ka}{{\cal A}}
\newcommand{\kb}{{\cal B}}

\newcommand{\ke}{{\cal E}}
\newcommand{\kf}{{\cal F}}
\newcommand{\kg}{{\cal G}}

\newcommand{\ko}{{\cal O}}

\newcommand{\ZZ}{\mathbb{Z}}
\newcommand{\QQ}{\mathbb{Q}}

\newcommand{\CC}{\mathbb{C}}

\newcommand{\PP}{\mathbb{P}}

\begin{document}

	\title[Non-uniqueness of Fourier--Mukai kernels]{Non-uniqueness of Fourier--Mukai kernels}

	\author{Alberto Canonaco and Paolo Stellari}

	\address{A.C.: Dipartimento di Matematica ``F. Casorati'', Universit{\`a}
	degli Studi di Pavia, Via Ferrata 1, 27100 Pavia, Italy}
	\email{alberto.canonaco@unipv.it}

	\address{P.S.: Dipartimento di Matematica ``F.
	Enriques'', Universit{\`a} degli Studi di Milano, Via Cesare Saldini
	50, 20133 Milano, Italy}
	\email{paolo.stellari@unimi.it}

	\keywords{Derived categories, Fourier-Mukai functors}

	\subjclass[2000]{14F05, 18E25, 18E30}
	
		\begin{abstract}
			We prove that the kernels of Fourier--Mukai
                        functors are not unique in general. On the
                        other hand we show that the cohomology sheaves
                        of those kernels are unique. We also discuss
                        several properties of the functor sending an
                        object in the derived category of the product
                        of two smooth projective schemes to the
                        corresponding Fourier--Mukai functor.
		\end{abstract}

		\maketitle

	\section{Introduction}\label{Intro}

All functors that appeared so far in the geometric applications of the
theory of derived categories have a very special nature: they are
\emph{Fourier--Mukai functors}.
Recall that if $X_1$ and $X_2$ are projective schemes, an exact functor
$\fun{F}:\Dp(X_1)\to\Db(X_2)$
is of Fourier--Mukai type 
if there exists $\ke\in\Db(X_1\times X_2)$ and an isomorphism of
exact functors $\fun{F}\iso\FM{\ke}$, where, denoting by $p_i:X_1\times
X_2\to X_i$ the natural projections, $\FM{\ke}:\Dp(X_1)\to\Db(X_2)$ is
the exact functor defined by
\[
\FM{\ke}:=\rd(p_2)_*(\ke\lotimes p_1^*(-)).
\]
Such a complex $\ke$ is called a \emph{kernel} of $\fun{F}$. Recall that the category $\Dp(X_i)$ of
perfect complexes is the full triangulated subcategory of the bounded
derived category of coherent sheaves $\Db(X_i):=\Db(\Coh(X_i))$
consisting of complexes which are quasi-isomorphic to bounded
complexes of locally free sheaves of finite type over $X_i$. Notice
that $\Dp(X_i)$ coincides with $\Db(X_i)$ if and only if $X_i$ is
regular.

There are many advantages of having a functor which is described in
terms of an object in the derived category of the product. Among them is
the study of the action of those functors on cohomology leading, for
example, to a description of the group of autoequivalences of special
projective varieties (see \cite{Or1}). As Fourier--Mukai equivalences act
also on Hochschild homology and cohomology one may also study
deformations of smooth projective varieties together with deformations
of equivalences between the corresponding bounded derived categories of
coherent sheaves.

\medskip

Despite the relevance of these functors, two important and basic
questions remain open:
\medskip
\begin{itemize}
\item[(Q1)] {\it Are all exact functors between the bounded derived
  categories of coherent sheaves on smooth projective varieties of
  Fourier--Mukai type?}
\medskip
\item[(Q2)] {\it Is the kernel of a Fourier--Mukai
  functor unique (up to isomorphism)?}
\end{itemize}
\medskip
Obviously, the same questions may be reformulated more generally in terms of perfect
complexes on projective schemes.

The best evidence that the answer to both questions could be positive
is due to some beautiful results of To\"en concerning
dg-categories. Indeed, in \cite{To} it is shown that all dg
(quasi-)functors between the dg-categories of perfect complexes on
smooth proper schemes 
are of Fourier--Mukai type. This result, combined with the conjecture by Bondal, Larsen and Lunts
in \cite{BLL} saying that all exact functors between the bounded derived
categories of coherent sheaves on smooth projective varieties should
be liftable to dg (quasi-)functors between the corresponding
dg-enhancements, would answer positively (Q1).

\medskip

Contrary to the exhaustive picture for dg-categories, the results concerning derived categories are more
fragmentary and essentially provide responses to (Q1) and (Q2) under
some assumptions on the functor. In the seminal paper \cite{Or1}
(together with \cite{BB}) Orlov solved completely the case of fully faithful functors between
the bounded derived categories of coherent sheaves on smooth
projective varieties. Indeed, he proves that these functors are all of Fourier--Mukai type with unique (up to isomorphism) kernel. Various generalizations to
quotient stacks and twisted categories were given by Kawamata in \cite{Ka} and by the authors in
\cite{CS} respectively. In particular, in \cite{CS} a condition much
weaker than fully faithfulness is required for a functor to be of
Fourier--Mukai type. More recently, a new approach involving
dg-categories has been proposed by Lunts and Orlov in \cite{LO}, where they deal
with the case of fully faithful functors between the derived
categories of perfect complexes on projective schemes. This approach
allows them to avoid some of the assumptions made by Ballard in \cite{Ba1}. In
\cite{CS1}, we extend further the results in \cite{LO} and study exact functors between supported derived
categories.

\medskip

Back to the questions above, the main result in this paper shows that
the answer to (Q2) cannot be positive in general (see Section
\ref{sect:uniqueness} for the proof).

\begin{thm}\label{thm:main1}
For every elliptic curve $X$ over an algebraically closed
field there exist $\ke_1,\ke_2\in\Db(X\times X)$ such that
$\ke_1\not\iso\ke_2$ but $\FM{\ke_1}\iso\FM{\ke_2}$.
\end{thm}

More precisely, we get the following picture. Given two smooth
projective varieties $X_1$ and $X_2$, denote by
$\ExFun(\Db(X_1),\Db(X_2))$ the category of exact functors between
$\Db(X_1)$ and $\Db(X_2)$. Putting all together, we will see in
Sections \ref{sect:prel} and \ref{sect:uniqueness} that the natural
functor
\begin{equation}\label{eqn:fun}
\FM[X_1\to X_2]{\farg}\colon\Db(X_1\times X_2)\longrightarrow
\ExFun(\Db(X_1),\Db(X_2))
\end{equation}
sending $\ke$ to the functor $\FM{\ke}=\FM[X_1\to X_2]{\ke}$ is, in
general, neither essentially injective (Theorem \ref{thm:main1}) nor
faithful (see \cite[Example 6.5]{C}) nor full (Proposition
\ref{nonfaith}). Moreover we cannot even expect that
$\ExFun(\Db(X_1),\Db(X_2))$ has a triangulated structure making the
above functor exact (Corollary \ref{cor:notria}). Such a negative
picture puts the optimistic hope to answer question (Q1) positively
a bit in the shade.

\medskip

On the positive side, in Section \ref{sect:uniqcohom} we prove the
following result, which provides our best substitute for the
uniqueness of Fourier--Mukai kernels.

\begin{thm}\label{thm:main2}
Let $X_1$ and $X_2$ be projective schemes and let
$\fun{F}\colon\Dp(X_1)\to\Db(X_2)$ be an exact functor. If
$\fun{F}\iso\FM{\ke}$ for some $\ke\in\Db(X_1\times X_2)$, then the
cohomology sheaves of $\ke$ are uniquely determined (up to
isomorphism) by $\fun{F}$.
\end{thm}

Notice that, as a consequence, the class in the Grothendieck group
$K(X_1\times X_2)$ of a Fourier--Mukai kernel is uniquely determined
by the functor.

\medskip

After the final version of this paper was completed, we were informed
that the example used in the proof of Theorem \ref{thm:main1} had
already been circulating among some people. As we could not find any
mention of this result in the literature, we still believe that it is
important to have it written down.

\medskip

\noindent{\bf Notation.} In the paper, $\K$ is a field and all schemes
are assumed to be over $\K$. Notice that in Sections \ref{sect:prel} and \ref{sect:uniqueness}, the field $\K$ is assumed to be algebraically closed. All additive (in particular,
triangulated) categories and all additive (in particular, exact)
functors will be assumed to be $\K$-linear. An additive category will
be called $\Hom$-finite if the $\K$-vector space $\Hom(A,B)$ is finite
dimensional for every objects $A$ and $B$. If $f\colon A\to B$ is a
morphism in a triangulated category, the \emph{cone} of $f$, denoted
by $\cone{f}$, is an object (well defined up to isomorphism) fitting
into a distinguished triangle $A\mor{f}B\to\cone{f}\to\sh{A}$.

\section{Properties of the functor $\FM[X_1\to X_2]{\farg}$}\label{sect:prel}

In this section we deal with some preliminary results concerning the
functor defined in \eqref{eqn:fun} from the derived category of the
product of two smooth projective varieties to the category of exact
functors between the corresponding derived categories of coherent
sheaves. The base field $\K$ is assumed to be algebraically closed.

\subsection{Counterexamples to faithfulness and fullness}\label{subsect:counter}
Following the notation in the introduction, if $\cat{T}_1$ and
$\cat{T}_2$ are two triangulated categories, we denote by
$\ExFun(\cat{T}_1,\cat{T}_2)$ the category whose objects are the exact
functors from $\cat{T}_1$ to $\cat{T}_2$ and whose morphisms are the
natural transformations compatible with shifts. Clearly
$\ExFun(\cat{T}_1,\cat{T}_2)$ is additive and has a natural shift
functor, but, due to the non-functoriality of the cone, it is not
known if in general it can be endowed with any triangulated structure. In particular, it is not expected to possess a natural one.

Now assume that $X_i$ for $i=1,2$ are two smooth projective varieties
of dimension $d_i$. It is easy to see that the map
$\Ker\mapsto\FM{\Ker}=\FM[X_1\to X_2]{\Ker}$ extends to the functor
\eqref{eqn:fun}, which is obviously additive and compatible with
shifts. It is natural to study properties of this functor, in
particular one can ask if it is faithful, full, essentially injective
(i.e.\ if a kernel of a Fourier--Mukai functor is unique up to
isomorphism), essentially surjective (i.e.\ if every exact functor is
of Fourier--Mukai type) or if $\ExFun(\Db(X_1),\Db(X_2))$ admits a
triangulated structure such that $\FM[X_1\to X_2]{\farg}$ is exact. We
are going to see that, at least for some choices of $X_1$ and $X_2$,
the answers to most of these questions are negative. Unfortunately we
were unable to prove anything new about essential surjectivity, which
is certainly a very intriguing problem.

\begin{remark}\label{exchange}
The functor $\FM[X_2\to X_1]{\farg}$ satisfies one of the properties
we are interested in if and only if $\FM[X_1\to X_2]{\farg}$ does:
this follows from the fact that $\FM[X_2\to X_1]{\farg}$ can be
identified with the opposite functor of $\FM[X_1\to X_2]{\farg}$ under
the equivalences $\Db(X_1\times X_2)\to\Db(X_1\times X_2)\opp$
(defined on the objects by $\Ker\mapsto\Ker\dual\otimes
p_1^*\sh[d_1]{\ds_{X_1}}$) and
$\ExFun(\Db(X_1),\Db(X_2))\to\ExFun(\Db(X_2),\Db(X_1))\opp$ (defined
on the objects by $\fun{F}\mapsto\fun{F}_*$, the right adjoint of
$\fun{F}$). Notice that we are using the fact that, in this context,
any exact functor has right and left adjoint by \cite{BB} (see also
\cite[Rmk.\ 2.1]{CS}).
\end{remark}

\begin{remark}
The functor $\FM[X_1\to X_2]{\farg}$ is an equivalence (hence it has
all the good properties we are investigating) if $d_1$ or $d_2$
is $0$. Indeed, by Remark \ref{exchange} we can assume $d_1=0$ (so
that $X_1=\spec\K$ is a point, being $\K$ algebraically closed), and then it is easy to see that a
quasi-inverse is the functor defined on objects by
$\fun{F}\mapsto\fun{F}(\K)$.
\end{remark}

So the interesting case to study is when $d_1,d_2>0$, but we can prove
something only when $d_1$ or $d_2$ is $1$. The reason for this is that
if $X$ is a smooth projective curve, then the abelian category
$\Coh(X)$ is hereditary (i.e.\ $\Ext^i(\s{F},\s{G})=0$ for every $i>1$
and for every $\s{F},\s{G}\in\Coh(X)$), which implies that every
object of $\Db(X)$ is isomorphic to the direct sum of its cohomology
sheaves. Being $X$ proper, by \cite[Thm.\ 2]{A}, the Krull--Schmidt theorem holds for the abelian category $\Coh(X)$. Namely, each object in $\Coh(X)$ can be written in a unique way (up to reordering and isomorphism) as a finite direct sum of indecomposable objects.
Moreover, being $X$ a smooth curve, every indecomposable object in $\Coh(X)$ is either a vector bundle or a torsion sheaf of the form $\so_{n\p}$ with $n$ a positive integer and $\p$ a closed point of $X$.
Since a natural transformation between additive functors is
always additive in the obvious sense, we see in particular that a
natural transformation of exact functors from $\Db(X)$ is determined
by its values on the indecomposable objects of $\Coh(X)$. This property is
essential in the proof of the following result, whose statement about
non-faithfulness is a generalization of \cite[Example 6.5]{C} (where only the particular case in which $X_1=X_2$ is an elliptic
curve is considered).

\begin{prop}\label{nonfaith}
If $\min\{d_1,d_2\}=1$, then $\FM[X_1\to X_2]{\farg}$ is neither
faithful nor full.
\end{prop}

\begin{proof}
By Remark \ref{exchange} we can assume that $1=d_1\le d_2$. Choose a
finite morphism $f\colon X_1\to\PP^{d_2}$ and a finite and surjective
(hence flat) morphism $g\colon X_2\to\PP^{d_2}$. Then
$\fun{F}:=g^*\comp f_*\colon\Coh(X_1)\to\Coh(X_2)$ is an exact
functor, which trivially extends to an exact functor again denoted by
$\fun{F}\colon\Db(X_1)\to\Db(X_2)$. Clearly there exists
$0\niso\Ker\in\Db(X_1\times X_2)$ such that $\fun{F}\iso\FM{\Ker}$.

In order to prove that $\FM[X_1\to X_2]{\farg}$ is not faithful,
notice that, by Serre duality,
\[\Hom_{\Db(X_1\times X_2)}(\Ker,\Ker)\iso\Hom_{\Db(X_1\times X_2)}
(\Ker,\Ker\otimes\sh[1+d_2]{\ds_{X_1\times X_2}})\dual,\]
so there exists $0\ne\alpha\in\Hom_{\Db(X_1\times X_2)}
(\Ker,\Ker\otimes\sh[1+d_2]{\ds_{X_1\times X_2}})$. Since
$\ds_{X_1\times X_2}\iso p_1^*\ds_{X_1}\otimes p_2^*\ds_{X_2}$, this induces for
any $\s{F}\in\Coh(X_1)$ a morphism
\[\FM{\alpha}(\s{F})\colon\FM{\Ker}(\s{F})\iso\fun{F}(\s{F})\to
\FM{\Ker\otimes\sh[1+d_2]{\ds_{X_1\times X_2}}}(\s{F})\iso
\fun{F}(\s{F}\otimes\ds_{X_1})\otimes\sh[1+d_2]{\ds_{X_2}}.\]
As $\fun{F}(\s{F})$ and $\fun{F}(\s{F}\otimes\ds_{X_1})$ are objects
of $\Coh(X_2)$, it follows that $\FM{\alpha}(\s{F})=0$, whence
$\FM{\alpha}=0$.

Now we are going to show that $\FM[X_1\to X_2]{\farg}$ is not full. We
start by observing that for every closed point $\p\in X_1$ we can define
a natural transformation $\zeta_{\p}\colon\id\to\sh$ of exact functors
on $\Db(X_1)$ by setting $\zeta_{\p}(\s{F}):=0$ for every
indecomposable object of $\Coh(X_1)$ not isomorphic to $\so_{\p}$ and
taking $\zeta_{\p}(\so_{\p})\ne0$ (note that the latter is an element
of $\Hom(\so_{\p},\sh{\so_{\p}})\iso\Hom(\so_{\p},\so_{\p})\dual\iso\K$
by Serre duality), and then extending additively and by shifts in the
obvious way. It is easy to see that in this way $\zeta_{\p}$ is really
a natural transformation, namely that
$\sh{\phi}\comp\zeta_{\p}(\s{F})=\zeta_{\p}(\s{G})\comp\phi$ for
every morphism $\phi\colon\s{F}\to\s{G}$ in $\Db(X)$: indeed, it is
enough to assume that $\s{F},\s{G}\in\Coh(X)$ are
indecomposable, in which case the required equality follows from
Lemma \ref{fac} below if $\s{F}$ and $\s{G}$ are supported at $\p$,
and is otherwise trivial. (Indeed, we use Lemma \ref{fac} identifying $F_n$ with $\so_{n\p}$, as it is explained in the paragraph before the lemma.)

Composing with $\fun{F}$ clearly
defines a natural transformation from $\fun{F}\iso\FM{\Ker}$ to
$\sh{\fun{F}}\iso\FM{\sh{\Ker}}$, hence an element $\zeta'_{\p}\in
\Hom_{\ExFun(\Db(X_1),\Db(X_2))}(\FM{\Ker},\FM{\sh{\Ker}})$. It is not
difficult to see that $\zeta'_{\p}(\so_{\p})\ne0$, which implies that
\[\dim_{\K}\Hom_{\ExFun(\Db(X_1),\Db(X_2))}
(\FM{\Ker},\FM{\sh{\Ker}})=\infty,\]
thereby proving that $\FM[X_1\to X_2]{\farg}$ is not full.
\end{proof}

\subsection{Projective line}\label{subsect:PP^1}

We start by proving the uniqueness (up to isomorphism) of Fourier--Mukai kernels for the projective line. This has to be compared with
the more interesting case of elliptic curves (Section
\ref{sect:uniqueness}).

\begin{prop}\label{essinj}
If $X_1$ or $X_2$ is $\PP^1$, then $\FM[X_1\to X_2]{\farg}$ is
essentially injective.
\end{prop}

\begin{proof}
As usual, by Remark \ref{exchange} we can assume that
$X_1=\PP^1$. Since on $\PP^1\times\PP^1$ there is a resolution of the
diagonal of the form
\[0\to\so(-1,-1)\mor{x_0\boxtimes x_1-x_1\boxtimes x_0}\so\to\sod\to0,\]
the argument in \cite[Sect.\ 4.3]{CS} shows that, for every exact
functor $\fun{F}\colon\Db(\PP^1)\to\Db(X_2)$, any object $\Ker$ in
$\Db(\PP^1\times X_2)$ such that $\fun{F}\iso\FM{\Ker}$ is necessarily
a convolution of the complex
\[\so(-1)\boxtimes\fun{F}(\so(-1))
\mor{\varphi:=x_0\boxtimes\fun{F}(x_1)-x_1\boxtimes\fun{F}(x_0)}
\so\boxtimes\fun{F}(\so),\]
hence it is uniquely determined up to isomorphism as the cone of
$\varphi$.
\end{proof}

We conclude this section showing that, in some cases, the category
$\ExFun(\Db(X_1),\Db(X_2))$ cannot have a suitable triangulated
structure. For this we need some preliminary results.

\begin{lem}\label{trivcone}
Let $\cat{T}$ be a $\Hom$-finite triangulated category and let $f\colon
A\to B$ be a morphism of $\cat{T}$. Then $\cone{f}\iso\sh{A}\oplus
B$ if and only if $f=0$.
\end{lem}

\begin{proof}
The other implication being well-known, we assume that
$\cone{f}\iso\sh{A}\oplus B$. Applying the cohomological functor
$\Hom(\farg,B)$ to the distinguished triangle $A\mor{f}B\to
\sh{A}\oplus B\to\sh{A}$, one gets an exact sequence of finite
dimensional $\K$-vector spaces
\[\Hom(\sh{A},B)\to\Hom(\sh{A}\oplus B,B)\to\Hom(B,B)
\mor{(\farg)\comp f}\Hom(A,B).\]
For dimension reasons, the last map must be $0$, hence $f=0$.
\end{proof}

\begin{lem}\label{injfaith}
Let $\fun{F}\colon\cat{T}\to\cat{T}'$ be an exact functor between
triangulated categories, and assume that $\cat{T}$ is
$\Hom$-finite. If $\fun{F}$ is essentially injective, then $\fun{F}$
is faithful, too.
\end{lem}

\begin{proof}
Let $f\colon A\to B$ be a morphism of $\cat{T}$ such that
$\fun{F}(f)=0$. Then
\[\fun{F}(\cone{f})\iso\cone{\fun{F}(f)}\iso
\sh{\fun{F}(A)}\oplus\fun{F}(B)\iso\fun{F}(\sh{A}\oplus B)\] in
$\cat{T}'$, whence $\cone{f}\iso\sh{A}\oplus B$ in $\cat{T}$ because
$\fun{F}$ is essentially injective. It follows from Lemma
\ref{trivcone} that $f=0$.
\end{proof}

\begin{cor}\label{cor:notria}
If $d_1,d_2>0$ and $X_1$ or $X_2$ is $\PP^1$, then
$\ExFun(\Db(X_1),\Db(X_2))$ does not admit a triangulated structure
such that $\FM[X_1\to X_2]{\farg}$ is exact.
\end{cor}

\begin{proof}
This follows from Lemma \ref{injfaith}, since we know that in this case
$\FM[X_1\to X_2]{\farg}$ is essentially injective by Proposition
\ref{essinj}, but not faithful by Proposition \ref{nonfaith}.
\end{proof}

\section{Elliptic curves and non-uniqueness}\label{sect:uniqueness}

In this section we provide the proof of Theorem \ref{thm:main1}, so we
assume that $\K$ is algebraically closed and that $X$ is an elliptic
curve. Up to replacing $\Db(X)$ with an equivalent category, we can
assume that there is exactly one object in every isomorphism class,
and more precisely, as explained in Section \ref{subsect:counter}, that every object is a finite direct sum of shifts
of coherent sheaves and that every object of $\Coh(X)$ is uniquely
(up to reordering) a finite direct sum of indecomposable
sheaves. Recall that the indecomposable objects of $\Coh(X)$ are
either vector bundles or torsion sheaves of the form $\so_{n\p}$ with
$n>0$ and $\p$ a closed point of $X$. The following result summarizes some properties of
indecomposable vector bundles over an elliptic curve.

\begin{prop}{\bf(\cite{T})}\label{Prop:T}
For every $r>0$ and $d\in\ZZ$ there is an indecomposable vector bundle
$E_{r,d}$ of rank $r$ and degree $d$ on $X$ such that:
\begin{itemize}
\item[\rm (i)] All indecomposable vector bundles of rank $r$ and
degree $d$ are those of the form $E_{r,d}\otimes L$ with
$L\in\Pic^0(X)$, and they are all distinct;
\item[\rm (ii)] If $k>0$,then $F_k:=E_{k,0}$ is the only
indecomposable vector bundle of rank $k$ and degree $0$ having global
sections (in particular, $F_1=\so_X$), and if $k>1$ there is an exact
sequence
\begin{equation}\label{Fseq}
0\to F_1\to F_k\to F_{k-1}\to0;
\end{equation}
\item[\rm (iii)] If $n=\gcd(r,d)$, $E_{r,d}=E_{r/n,d/n}\otimes F_n$;
\item[\rm (iv)] $E_{r,d}$ (hence also $E_{r,d}\otimes L$ for every
$L\in\Pic^0(X)$) is semistable, and it is stable if and only if
$\gcd(r,d)=1$.
\end{itemize}
\end{prop}

\begin{cor}\label{indec}
Let $E_i$ (for $i=1,2$) be indecomposable vector bundles of rank $r_i$
and degree $d_i$ on $X$ with the property that
$\Hom(E_1,E_2)\ne0\ne\Hom(E_2,E_1)$. Then, setting
$n_i:=\gcd(r_i,d_i)$, there exists a stable vector bundle $E$ of rank
$r_1/n_1=r_2/n_2$ and degree $d_1/n_1=d_2/n_2$ such that $E_i=E\otimes
F_{n_i}$ for $i=1,2$.
\end{cor}

\begin{proof}
The hypothesis, together with the fact that $E_1$ and $E_2$ are
semistable, implies that $d_1/r_1=d_2/r_2$, from which it is immediate
to deduce that $r_1/n_1=r_2/n_2$ and $d_1/n_1=d_2/n_2$. Set $r:=r_1/n_1$ and $d:=d_1/n_1$. As
$E_i=E_{r_i,d_i}\otimes L_i$ for some $L_i\in\Pic^0(X)$, we get
$E_i=E_{r,d}\otimes L_i\otimes F_{n_i}$ for $i=1,2$ (see parts (i) and (iii) of Proposition \ref{Prop:T}). It remains to
prove that $L_1=L_2$, because then we can conclude setting
$E:=E_{r,d}\otimes L_i$ (which, due to part (iv) of Proposition \ref{Prop:T}, is stable since
$\gcd(r,d)=1$). Assuming instead that $L_1\ne L_2$, we will reach a
contradiction by showing that $\Hom(E_1,E_2)=0$. We proceed by
induction on $n_1+n_2$: the case $n_1=n_2=1$ follows from the fact
$E_i=E_{r,d}\otimes L_i$ for $i=1,2$ are distinct stable vector
bundles (see parts (i) and (iv) of Proposition \ref{Prop:T}). As for the inductive step, we suppose $n_2>1$ (the case
$n_1>1$ is similar) and apply the functor $\Hom(E_1,E_{r,d}\otimes
L_2\otimes\farg)$ to \eqref{Fseq} with $k=n_2$. This yields an exact
sequence
\[\Hom(E_1,E_{r,d}\otimes L_2\otimes F_1)\to\Hom(E_1,E_{r,d}\otimes L_2\otimes F_{n_2})\to
\Hom(E_1,E_{r,d}\otimes L_2\otimes F_{n_2-1}).\] By induction, the first and the third terms in the sequence are $0$, whence
the second one is $0$ as well. But $E_{r,d}\otimes L_2\otimes F_{n_2}=E_2$ and this provides the desired contradiction.
\end{proof}

We will denote by $\cat{T}_E$ for $E$ a stable vector bundle on $X$
(respectively $\cat{T}_{\p}$ for $\p$ a closed point of $X$) the full
triangulated subcategory of $\Db(X)$ classically generated by $E$
(respectively $\so_{\p}$), namely the smallest strictly full triangulated subcategory of $\Db(X)$ containing $E$
(respectively $\so_{\p}$) and closed under direct summands. Since
$\rd\Hom(E,E)\iso\rd\Hom(\so_{\p},\so_{\p})\iso\K\oplus\sh{\K}$ (so
that $E$ and $\so_{\p}$ are $1$-spherical objects), it follows from
\cite[Thm.\ 2.1]{KYZ} that these categories are all equivalent; it
is also clear that the indecomposable sheaves of $\cat{T}_E$
(respectively $\cat{T}_{\p}$) are $E\otimes F_n$ (respectively
$\so_{n\p}$) for $n>0$. In the following we will also denote by
$\cat{T}$ any of the equivalent categories $\cat{T}_E$ or
$\cat{T}_{\p}$, but for simplicity of notation we will identify it
with $\cat{T}_{\so_X}$.  As $\cat{T}$ is equivalent to a (derived)
category of $\K[x]$--modules, $F_n$ corresponding to $\K[x]/(x^n)$
(this is perhaps easier to see regarding $\cat{T}$ as $\cat{T}_{\p}$),
it is clear that $\dim_{\K}\Hom(F_m,F_n)=\min\{m,n\}$, for $m,n>0$, and
there are (non-split) distinguished triangles in $\cat{T}$ (the second
one is induced by \eqref{Fseq} with $k=n+1$)
\begin{gather}\label{tp1}
F_n\mor{\pi'_{n+1,1}}F_{n+1}\mor{\pi_{n+1,1}}F_1
\mor{\pi''_{n+1,1}}\sh{F_n} \\
\label{tpn}
F_1\mor{\pi'_{n+1,n}}F_{n+1}\mor{\pi_{n+1,n}}F_n
\mor{\pi''_{n+1,n}}\sh{F_1}
\end{gather}
where $\pi_{m,n}\colon F_m\to F_n$ for $m>n$ denotes the
natural projection.

\begin{lem}\label{fac}
If $0<m\le n$, every morphism $F_m\to F_{n+1}$ factors through
$\pi'_{n+1,1}$ and every morphism $F_{n+1}\to F_m$ factors through
$\pi_{n+1,n}$. Moreover, every morphism $F_{n+1}\to F_{n+1}$ is
uniquely the sum of $\lambda\id$ for some $\lambda\in\K$ and of a
morphism which factors through $\pi'_{n+1,1}$ and $\pi_{n+1,n}$.
\end{lem}

\begin{proof}
By \eqref{tp1} the map $\Hom(F_m,F_n)
\mor{\pi'_{n+1,1}\comp(-)}\Hom(F_m,F_{n+1})$ is injective for
every $m>0$. In particular, if $m\le n$ the map is also surjective
because both spaces have dimension $m$, whereas if $m=n+1$ the first
space has dimension $n$, the second $n+1$ and clearly $\id$ is not in
the image of the map. This proves both statements involving
$\pi'_{n+1,1}$, and those involving $\pi_{n+1,n}$ can be proved in a
similar way using \eqref{tpn} instead of \eqref{tp1}.
\end{proof}

Since $\Hom(\sh[-1]{\sod},\sh{\sod})\iso\Hom(\sod,\sod)\dual\iso\K$ by
Serre duality, an object $\Ker$ of $\Db(X\times X)$ obtained as the
cone of a nonzero morphism $\sh[-1]{\sod}\to\sh{\sod}$ is well defined
up to isomorphism. Notice that such a morphism is the same as the one considered in \cite[Example 6.5]{C}. Setting also $\Ker_0:=\sod\oplus\sh{\sod}$, we have
$\Ker\niso\Ker_0$ by Lemma \ref{trivcone}. We are going to prove
that
\[\FM{\Ker}\iso\FM{\Ker_0}\colon\Db(X)\to\Db(X).\]
To this purpose, we start by observing that $\FM{\Ker_0}$ is just
$\id\oplus\sh{}$ and that $\FM{\Ker}$ coincides with $\FM{\Ker_0}$ on
objects. As explained in the following example, in general, this is not enough to conclude that the two functors are isomorphic.

\begin{ex}
An easy calculation shows that, on $\PP^1\times\PP^1$, there is an isomorphism of $\K$-vector spaces
$\Hom(\sh[-1]{\Delta_*\ko_{\PP^1}},\sh{\Delta_*(\omega_{\PP^1}^{\otimes 2})})\iso\K$.
Take $0\neq f\in\Hom(\sh[-1]{\Delta_*\ko_{\PP^1}},\sh{\Delta_*(\omega_{\PP^1}^{\otimes 2})})$ and consider the objects $\kf_0:=\Delta_*\ko_{\PP^1}\oplus\sh{\Delta_*(\omega_{\PP^1}^{\otimes 2})}$ and $\kf:=\cone{f}$ in $\Db(\PP^1\times\PP^1)$. Obviously $\FM{\kf_0}$ and $\FM{\kf}$ coincide on objects because $\kg\oplus\sh{(\kg\otimes\omega_{\PP^1}^{\otimes 2})}\iso\FM{\kf_0}(\kg)\iso\FM{\kf}(\kg)$, for every $\kg\in\Db(\PP^1)$. On the other hand $\kf_0\not\iso\kf$ (use again Lemma \ref{trivcone}) and so, by Proposition \ref{essinj}, the functors $\FM{\kf_0}$ and $\FM{\kf}$ are not isomorphic.
\end{ex}

Back to the genus $1$ case, to prove that $\FM{\Ker}\iso\FM{\Ker_0}$ we have to take care of
morphisms as well. To this end observe that $\FM{\Ker}$ is defined on every morphism $f\colon \ka\to\kb$ of
$\Db(X)$ by
\[\fun{\FM{\Ker}}(f)=\begin{pmatrix}
f & 0 \\
\epsilon(f) & \sh{f}
\end{pmatrix}\colon\ka\oplus\sh{\ka}\to \kb\oplus\sh{\kb}\]
for some $\epsilon(f)\colon \ka\to\sh{\kb}$. Notice that $\epsilon$ is
$\K$--linear in the obvious sense (because $\FM{\Ker}$ is
$\K$--linear), $\epsilon(\id_\ka)=0$ for every object $\ka$ of $\Db(X)$
(because $\fun{\FM{\Ker}}(\id_\ka)=\id_{\FM{\Ker}(\ka)}$) and
$\epsilon(g\comp f)=\epsilon(g)\comp f+\sh{g}\comp\epsilon(f)$ for
every pair of composable morphisms $f$ and $g$ of $\Db(X)$ (because
$\fun{\FM{\Ker}}(g\comp f)=\fun{\FM{\Ker}}(g)\comp\FM{\Ker}(f)$). It
is also evident that if $\ka$
and $\kb$ are sheaves and $f\in\Hom(\ka,\sh{\kb})$, then $\epsilon(f)=0$.

\begin{prop}
With the above notation, there is an isomorphism
$\FM{\Ker}\iso\FM{\Ker_0}$.
\end{prop}

\begin{proof}
We will show that there is an isomorphism
$\eta\colon\FM{\Ker_0}\isomor\FM{\Ker}$ such that
\[\eta(\ka)=\begin{pmatrix}
\id & 0 \\
\beta(\ka) & \id
\end{pmatrix}\colon \ka\oplus\sh{\ka}\to \ka\oplus\sh{\ka}\]
for every object $\ka$ of $\Db(X)$. It is clearly enough to define
$\eta(\ka)$ for every indecomposable sheaf $\ka$ (and then extend
additively and by shifts in the obvious way) so that
$\FM{\Ker}(f)\comp\eta(\ka)=\eta(\kb)\comp\FM{\Ker_0}(f)$ for every
morphism of indecomposable sheaves $f\colon \ka\to \kb$. Now, this
equality is equivalent to $\epsilon(f)=\beta(\kb)\comp
f-\sh{f}\comp\beta(\ka)$, so it is certainly satisfied if either
$\Hom(\ka,\kb)=0$ or $\Hom(\ka,\sh{\kb})=0$. On the other hand, if the
indecomposable sheaves $\ka$ and $\kb$ are such that
$\Hom(\ka,\kb)\ne0\ne\Hom(\ka,\sh{\kb})$, then $\ka$ and $\kb$ belong
to the same subcategory $\cat{T}_E$ ($E$ a vector bundle) or
$\cat{T}_{\p}$ ($\p$ a closed point). This follows from Corollary
\ref{indec} if both $\ka$ and $\kb$ are vector bundles (taking into
account that $\Hom(\ka,\sh{\kb})\iso\Hom(\kb,\ka)\dual$ by Serre
duality), whereas it is trivial in the other cases.

So, setting $\fun{F}:=\FM{\Ker}\rest{\cat{T}}$ and
$\fun{F_0}=\FM{\Ker_0}\rest{\cat{T}}$, it is enough to prove that
there is an isomorphism $\eta\colon\fun{F}_0\to\fun{F}$ of the above
form. In order to do that, we are going to define inductively for
every $n>0$ (exact) functors $\fun{F}_n\colon\cat{T}\to\cat{T}$ and
morphisms of $\cat{T}$
\[\alpha_n=\begin{pmatrix}
\id & 0 \\
\beta_n & \id
\end{pmatrix}
\colon F_n\oplus\sh{F_n}\to F_n\oplus\sh{F_n}\]
with the following properties:
\begin{itemize}
\item[(a)] $\fun{F}_1=\fun{F}$ and $\alpha_1=\id$;
\item[(b)] for every $n>0$ the functor $\fun{F}_n$ coincides with
$\fun{F}_0$ on objects, $\fun{F}_n(f)=\begin{pmatrix}
f & 0 \\
\epsilon_n(f) & \sh{f}
\end{pmatrix}$ for every morphism $f$ of $\cat{T}$ and
$\fun{F}_n\rest{\cat{T}_n}=\fun{F}_0\rest{\cat{T}_n}$,
where $\cat{T}_n$ denotes the full additive and closed under shifts
(but not triangulated) subcategory of $\cat{T}$ generated by $F_i$
for $0<i\le n$;
\item[(c)] for every $n>1$ the morphisms $\eta_n(F_m)\colon
F_m\oplus\sh{F_m}\to F_m\oplus\sh{F_m}$ (for
$m>0$) defined by $\eta_n(F_n)=\alpha_n$ and
$\eta_n(F_m)=\id$ if $m\ne n$, extend to an isomorphism
$\eta_n\colon\fun{F}_n\isomor\fun{F}_{n-1}$.
\end{itemize}
Once this is done, it is then straightforward to check that the
morphisms $\eta(F_n):=\alpha_n$ (for $n>0$) extend to an
isomorphism $\eta\colon\fun{F}_0\isomor\fun{F}$ as wanted.

In order to perform the inductive step from $n$ to $n+1$, notice that
for an arbitrary choice of $\beta_{n+1}$ (hence of $\alpha_{n+1}$ and
of $\eta_{n+1}$) and setting
$\fun{F}_{n+1}(f):=\eta_{n+1}(B)^{-1}\comp\fun{F}_n(f)\comp\eta_{n+1}(A)$
for every morphism $f\colon A\to B$ of $\cat{T}$, all the required
properties are satisfied, except possibly
$\fun{F}_{n+1}\rest{\cat{T}_{n+1}}=\fun{F}_0\rest{\cat{T}_{n+1}}$. Since
$\fun{F}_{n+1}\rest{\cat{T}_n}=\fun{F}_n\rest{\cat{T}_n}$ by
construction and $\fun{F}_n\rest{\cat{T}_n}=\fun{F}_0\rest{\cat{T}_n}$
by the inductive hypothesis, in view of Lemma \ref{fac} this last
condition holds if and only if the diagram
\[\xymatrix{\fun{F}(F_n) \ar[rr]^-{\fun{F}_0(\pi'_{n+1,1})}
\ar[drr]_-{\fun{F}_n(\pi'_{n+1,1})\;\;\;} & & \fun{F}(F_{n+1})
\ar[d]^{\alpha_{n+1}} \ar[rr]^-{\fun{F}_0(\pi_{n+1,n})} & &
\fun{F}(F_n) \\
 & & \fun{F}(F_{n+1}) \ar[urr]_-{\;\;\;\fun{F}_n(\pi_{n+1,n})}
}\]
commutes. Clearly this is true if and only if
\begin{gather}\label{ep1}
\epsilon_n(\pi'_{n+1,1})=\beta_{n+1}\comp\pi'_{n+1,1}, \\
\label{epn}
\epsilon_n(\pi_{n+1,n})=-\sh{\pi_{n+1,n}}\comp\beta_{n+1}.
\end{gather}
As $\epsilon_n(\pi'_{n+1,1})\comp\sh[-1]{\pi''_{n+1,1}}\in
\Hom(\sh[-1]{F_1},\sh{F_{n+1}})=0$, from the distinguished
triangle \eqref{tp1} we deduce that there exists
$\beta_{n+1}\colon F_{n+1}\to\sh{F_{n+1}}$ such that
\eqref{ep1} is satisfied, and we claim that then \eqref{epn} is
automatically true, namely that $\gamma:=\epsilon_n(\pi_{n+1,n})+
\sh{\pi_{n+1,n}}\comp\beta_{n+1}=0$. Indeed, using \eqref{ep1} and the
fact that $\epsilon_n\rest{\cat{T}_n}=0$,
\begin{multline*}
\gamma\comp\pi'_{n+1,1}=\epsilon_n(\pi_{n+1,n})\comp\pi'_{n+1,1}
+\sh{\pi_{n+1,n}}\comp\beta_{n+1}\comp\pi'_{n+1,1} \\
=\epsilon_n(\pi_{n+1,n})\comp\pi'_{n+1,1}
+\sh{\pi_{n+1,n}}\comp\epsilon_n(\pi'_{n+1,1})=
\epsilon_n(\pi_{n+1,n}\comp\pi'_{n+1,1})=0.
\end{multline*}
It follows again from \eqref{tp1} that
$\gamma=\gamma'\comp\pi_{n+1,1}$ for some
$\gamma'\colon F_1\to\sh{F_n}$. Now, by Serre duality,
$\Hom(F_1,\sh{F_n})\iso\Hom(F_n,F_1)\dual\iso\K$,
so there exists $\lambda\in\K$ such that
$\gamma'=\lambda\pi''_{n+1,1}$, whence
$\gamma=\lambda\pi''_{n+1,1}\comp\pi_{n+1,1}=0$.
\end{proof}

\begin{cor}
For every elliptic curve $X$ the functor $\FM[X\to X]{\farg}$ is not
essentially injective.
\end{cor}

\section{The uniqueness of the cohomology sheaves}\label{sect:uniqcohom}

In this section we prove Theorem \ref{thm:main2}, hence we assume that
$X_1$ and $X_2$ are projective schemes with ample divisors $H_1$ and $H_2$ on $X_1$ and $X_2$ respectively.
For $l\in\ZZ$, denote by $\cat{C}_l$ the full
subcategory with objects $\{\ko_{X_1}(mH_1):m>l\}\subset\Coh(X_1)$. Consider
Fourier--Mukai functors
\[
\FM{\ke_1},\FM{\ke_2}:\Dp(X_1)\longrightarrow\Db(X_2)
\]
where $\ke_1,\ke_2\in\D(\Qcoh(X_1\times X_2))$, and such that there exists an isomorphism
\begin{equation}\label{eqn:iso1}
\beta\colon\FM{\ke_1}\rest{\cat{C}_l}\isomor\FM{\ke_2}\rest{\cat{C}_l},
\end{equation}
for some integer $l$.

The following easy lemma shows that we can be more precise about the
Fourier--Mukai kernels above.

\begin{lem}\label{lem:boundcoh}
Under the above assumptions, $\ke_i\in\Db(X_1\times X_2)$, for $i=1,2$. Conversely, any $\ke\in\Db(X_1\times X_2)$ yields a Fourier--Mukai functor $\FM{\ke}\colon\Dp(X_1)\to\Db(X_2)$.
\end{lem}

\begin{proof}
The second part of the statement is clear. For the first one, we can apply the argument in \cite[Cor.\ 9.13 (4)]{LO} where the assumption that $\FM{\ke_i}$ is fully faithful is not used.

For the convenience of the reader, we provide a different easy argument.
Indeed, due to \cite[Lemma 7.47]{R}, $\ke_i\in\Db(X_1\times X_2)$ if and only
if, for all $\kf\in\Dp(X_1\times X_2)$, we have
\[
\dim\bigoplus_j\Hom(\kf,\sh[j]{\ke_i})<\infty.
\]
Let $G_i$ be a compact generator of $\D(\Qcoh(X_i))$, for $i=1,2$. By
\cite[Lemma 3.4.1]{BB}, $G_1\boxtimes G_2$ is a compact generator of
$\D(\Qcoh(X_1\times X_2))$ (see \cite{BB,R} for the definition of compact generator). As $\bigoplus_j\Hom(G_1\boxtimes
G_2,\sh[j]{\ke_i})\iso\bigoplus_j\Hom(G_2,\sh[j]{\FM{\ke_i}(G_1\dual)})$ is
finite dimensional because $\FM{\ke_i}(G_1\dual)\in\Db(X_2)$, we can
conclude using the fact that $G_1\boxtimes G_2$ classically generates
$\Dp(X_1\times X_2)$ (see, for example, \cite[Thm.\ 2.1.2]{BB}).
\end{proof}

The first step in the proof of of Theorem \ref{thm:main2} is the following.

\begin{lem}\label{lem:sheaves}
If $\ke_1,\ke_2\in\Coh(X_1\times X_2)$, then $\ke_1\iso\ke_2$.
\end{lem}

\begin{proof}
By \cite[Thm.\ 3.4.4]{EGA2}, for $i=1,2$, there is an isomorphism
between $\ke_i$ and the sheaf associated to
$M_i:=\bigoplus_{m\in\ZZ}(p_2)_*(\ke_i\otimes p_1^*\so_{X_1}(mH_1))$,
where $(p_2)_*$ is not derived. Since for $m\gg 0$ there are
functorial isomorphisms
\[
(p_2)_*(\ke_i\otimes p_1^*\so_{X_1}(mH_1))\iso
\FM{\ke_i}(\so_{X_1}(mH_1)),
\]
by \eqref{eqn:iso1} the graded modules $M_1$ and $M_2$ are isomorphic
in sufficiently high degrees. Hence, taking the associated sheaves, we
get $\ke_1\iso\ke_2$.
\end{proof}

If the Fourier--Mukai kernels are not sheaves, we have the following
result concerning their cohomologies. Notice that due to the weaker assumptions on the functors $\FM{\ke_1}$ and $\FM{\ke_2}$ in \eqref{eqn:iso1}, this may be seen as a stronger version of Theorem \ref{thm:main2}.

\begin{prop}\label{prop:isomcohom1}
For any $j\in\ZZ$, we have isomorphisms $H^j(\ke_1)\iso H^j(\ke_2)$ in
$\Coh(X_1\times X_2)$.
\end{prop}

\begin{proof}
We first prove that, given $j\in\ZZ$, we have $H^j(\ke_1)=0$ if and only if $H^j(\ke_2)=0$. Indeed, observe that $H^j(\ke_i)=0$ if and only if
$\Hom(\ko_{X_1}(mH_1)\boxtimes\ko_{X_2}(mH_2),\sh[j]{\ke_i})=0$ for $m\ll
0$. But
\begin{multline*}
\Hom(\ko_{X_1}(mH_1)\boxtimes\ko_{X_2}(mH_2),\sh[j]{\ke_1})\iso
\Hom(\ko_{X_2}(mH_2),\sh[j]{\FM{\ke_1}(\ko_{X_1}(-mH_1))})\\
\iso\Hom(\ko_{X_2}(mH_2),\sh[j]{\FM{\ke_2}(\ko_{X_1}(-mH_1))})
\iso\Hom(\ko_{X_1}(mH_1)\boxtimes\ko_{X_2}(mH_2),\sh[j]{\ke_2}).
\end{multline*}

We are now ready to prove the statement by induction on the number of
non-trivial cohomologies. If $\ke_1$ and $\ke_2$ are the shift of a sheaf, we can just apply Lemma \ref{lem:sheaves}. Thus assume that $\ke_i$ has at least two non-trivial cohomologies and that the last non-trivial
one is in degree $n$ while the first non-trivial one is in degree $n'<n$. In particular, we
have distinguished triangles
\begin{equation}\label{eqn:triacoh1}
\begin{array}{c}
\ke'_1\longrightarrow\ke_1\longrightarrow\sh[-n]{H^n(\ke_1)} \\
\ke'_2\longrightarrow\ke_2\longrightarrow\sh[-n]{H^n(\ke_2)},
\end{array}
\end{equation}
where $\ke'_1$ and $\ke'_2$ have cohomologies concentrated in the
interval $[n',n-1]$ which is strictly smaller than the one of $\ke_1$ and $\ke_2$.

Now observe that if $\ke\in\Db(X_1\times X_2)$ is such that $H^j(\ke)=0$ if $j\not\in[a,b]$, then we have $H^j(\FM{\ke}(\ko_{X_1}(mH_1)))=0$ if $j\not\in[a,b]$ and $m\gg 0$. Indeed,
\[
0=\Hom(\ko_{X_1}(-mH_1)\boxtimes\ko_{X_2}(-m'H_2),\sh[j]{\ke})\iso\Hom(\ko_{X_2}(-m'H_2),\sh[j]{\FM{\ke}(\ko_{X_1}(mH_1))}),
\]
if $m'\gg0$ and under the above assumptions on $m$ and $j$.

For $m\gg0$ from \eqref{eqn:triacoh1} we get the diagram
\begin{equation}\label{eqn:triacoh2}
\xymatrix{
\FM{\ke'_1}(\ko_{X_1}(mH_1))\ar[r]&\FM{\ke_1}(\ko_{X_1}(mH_1))\ar[r]\ar[d]^{\beta_m}& \sh[-n]{\FM{H^n(\ke_1)}(\ko_{X_1}(mH_1))} \\
\FM{\ke'_2}(\ko_{X_1}(mH_1))\ar[r]&\FM{\ke_2}(\ko_{X_1}(mH_1))\ar[r]&\sh[-n]{\FM{H^n(\ke_2)}(\ko_{X_1}(mH_1))},
}
\end{equation}
where the two rows are distinguished triangles and $\beta_m:=\beta(\ko_{X_1}(mH_1))$ is the
isomorphism induced by \eqref{eqn:iso1}.

Using the remark above, we get that $\FM{\ke'_1}(\ko_{X_1}(mH_1))$ has non-trivial cohomologies concentrated in degrees $[n',n-1]$ while $\sh[-n]{\FM{H^n(\ke_2)}(\ko_{X_1}(mH_1))}$ is a sheaf in degree $n$. Hence
\[
\Hom(\sh[k]{\FM{\ke'_1}(\ko_{X_1}(mH_1))},\sh[-n]{\FM{H^n(\ke_2)}(\ko_{X_1}(mH_1))})=0,
\]
for $m\gg 0$ and $k=0,1$. It follows that \eqref{eqn:triacoh2} can be completed to a
commutative diagram in a unique way. Thus, for some $l'>l$, we get natural
transformations
$\alpha:\FM{\ke'_1}\rest{\cat{C}_{l'}}\isomor\FM{\ke'_2}\rest{\cat{C}_{l'}}$ and
$\gamma:\FM{H^n(\ke_1)}\rest{\cat{C}_{l'}}\isomor\FM{H^n(\ke_2)}\rest{\cat{C}_{l'}}$,
which are easily seen to be isomorphisms (applying the same argument
to $\beta_m^{-1}$). By Lemma \ref{lem:sheaves}, we have
$H^n(\ke_1)\iso H^n(\ke_2)$ and, by induction, $H^j(\ke'_1)\iso
H^j(\ke'_2)$, for all $j\in\ZZ$. This is enough, as $H^j(\ke_i)\iso
H^j(\ke'_i)$, for $j<n$.
\end{proof}

Denoting by $K(X_1\times X_2)$ the Grothendieck group of the abelian
category $\Coh(X_1\times X_2)$, we clearly get the following result.

\begin{cor}\label{cor:Groth}
Let $X_1$ and $X_2$ be projective schemes. Consider two isomorphic
Fourier--Mukai functors
\[
\FM{\ke_1}\iso\FM{\ke_2}:\Dp(X_1)\longrightarrow\Db(X_2)
\]
Then $H^j(\ke_1)\iso H^j(\ke_2)$ for all $j\in\ZZ$. In particular,
$[\ke_1]=[\ke_2]$ in $K(X_1\times X_2)$.
\end{cor}

Notice that, if $\K=\CC$, it is well-known and very easy to see, under
the assumptions of Corollary \ref{cor:Groth}, that
$\mathrm{ch}(\ke_1)=\mathrm{ch}(\ke_2)\in H^*(X_1\times
X_2,\QQ)$. Indeed $\FM{\ke_i}$ induces a correspondence between
$H^*(X_1,\QQ)$ and $H^*(X_2,\QQ)$ given by the object
$\mathrm{ch}(\ke_i)\cdot\sqrt{\mathrm{td}(X_1\times X_2)}$ (see
\cite{Or1}). The K\"unneth decomposition for the cohomology of the
product yields then $\mathrm{ch}(\ke_1)=\mathrm{ch}(\ke_2)$.


\bigskip

{\small\noindent {\bf Acknowledgements.} The authors would like to thank Emanuele Macr\`i for interesting conversations. We are also grateful to Arend Bayer, Daniel Huybrechts, Dmitri Orlov and Pawel Sosna for comments on an early version of this paper. Part of this article was written while P.S.\ was visiting the Department of Mathematics of the University of Utah and the Institut Henri Poincar\'e in Paris whose warm hospitality is gratefully acknowledged. It is a pleasure to thank these institutions and the Istituto Nazionale di Alta Matematica for financial support.}


\end{document}